\documentclass[fleqn,preprint,3p,a4paper]{elsarticle}

\usepackage{amssymb}
\usepackage{amsmath}
\usepackage{amsthm}
\usepackage{dcolumn}
\usepackage{endnotes}
\usepackage{tabularx}
\usepackage[matrix,arrow]{xy}
\usepackage{wasysym}
\usepackage{xcolor}

\newtheorem{theorem}{Theorem}[section]
\newtheorem{proposition}[theorem]{Proposition}
\newtheorem{lemma}[theorem]{Lemma}
\newtheorem{corollary}[theorem]{Corollary}

\theoremstyle{definition}
\newtheorem{definition}[theorem]{Definition}
\newtheorem{example}[theorem]{Example}
\newtheorem{remark}[theorem]{Remark}

\newcommand{\ir}{{\mathsf{Irr}}}

\newcommand{\mn}{\mathbb N}

\newcommand{\cl}{{\rm cl}}
\newcommand{\ii}{{\rm int}}

\newcommand{\ua}{\mathord{\uparrow}}
\newcommand{\da}{\mathord{\downarrow}}

\newcommand{\mk}{\mathord{\mathsf{K}}}

\journal{Mathematical Structures in Computer Science}

\begin{document}

\begin{frontmatter}

\title{Strong well-filteredness of upper topology on sup-complete posets\tnoteref{t1}}
\tnotetext[t1]{This research was supported by the National Natural Science Foundation of China (Nos. 12471070, 12071199).}

\author[X. Xu]{Xiaoquan Xu\corref{mycorrespondingauthor}}
\cortext[mycorrespondingauthor]{Corresponding author}
\ead{xiqxu2002@163.com}
\address[X. Xu]{School of Mathematics and Statistics, Jiangxi Normal University, Nanchang 330000, China}
\author[Y. Yang]{Yi Yang}
\address[Y. Yang]{School of Mathematics and Statistics, Jiangxi Normal University, Nanchang 330000, China}
\ead{1354848335@qq.com}
\author[LZ. Chen]{Lizi Chen}
\address[LZ. Chen]{School of Mathematics and Statistics, Minnan Normal University, Zhangzhou 363000, China}
\ead{1403391465@qq.com}

\begin{abstract}
We first introduce and investigate a new class of $T_0$ spaces --- strong R-spaces, which are stronger than both R-spaces and strongly well-filtered spaces. It is proved that any sup-complete poset equipped with the upper topology is a strong R-space and the Hoare power space of a $T_0$-space is a strong R-space. Hence the upper topology on a sup-complete poset is strongly well-filtered and the Hoare power space of a $T_0$-space is strongly well-filtered, which answers two problems recently posed by Xu.
\end{abstract}


\end{frontmatter}

\section{Introduction}

In domain theory and non-Hausdorff topology, we encounter numerous links between topology and order theory (cf. \cite{Abramsky-Jung-1994, GHKLMS-2003, Goubault-2013}). The upper topology and Scott topology are probably two of the most important order-compatible topologies on posets, and sobriety is probably the
most important and useful property of $T_0$-spaces (see \cite{Abramsky-Jung-1994, GHKLMS-2003, Goubault-2013, Jia-2018, Xu-Zhao-2021}). The Hofmann-Mislove Theorem reveals a very
distinct characterization for the sober spaces via open filters and illustrates the close relationship
between domain theory and topology. With the development of domain theory and non-Hausdorff topology, another two properties
also emerged as the very useful and important properties for $T_0$-spaces: the property of being a
$d$-space and the well-filteredness (see \cite{Abramsky-Jung-1994, GHKLMS-2003, Goubault-2013, Jia-2018, Keimel-Lawson-2009, Wyler-1981, Xu-Shen-Xi-Zhao-2020-2, Xu-Zhao-2021}).

In order to uncover more finer links between $d$-spaces and well-filtered spaces, Xu and Zhao \cite{Xu-Zhao-2020} introduced and discussed strong $d$-spaces, and Xu \cite{Xu-2025} introduced and studied strongly well-filtered spaces. The strong $d$-spaces and strongly well-filtered spaces possess many important properties (see \cite{Li-Jin-Miao-Chen-2022, Lawson-Xu-2024-1, Lawson-Xu-2024-2, Xu-2025, Xu-Zhao-2020, Xu-Zhao-2021}), two of the key ones being that the Scott space $\Sigma~\!\!P$ of a poset $P$ is a strong $d$-space iff it is strongly well-filtered, and if $P$ is a dcpo such that $\Sigma P$ is strongly well-filtered and $\Sigma (P\times P)=\Sigma P\times\Sigma P$, then $\Sigma P$ is sober (see \cite[Theorem 4.10 and Theorem 4.14]{Xu-2025}).

Power constructions are very important structures in domain theory, which play a fundamental role in modeling the semantics of non-deterministic programming languages. The Hoare power construction and Smyth power construction are two main power constructions (see  \cite{Heckmann-1990, Heckmann-1992, Smyth-1978}). These power constructions were traditionally given in the category of dcpos, but are easily and profitably extended to general topological spaces (see, for example, \cite[Sections 6.2.3 and 6.2.4]{Abramsky-Jung-1994} and \cite{Schalk-1993}).

In \cite{Xu-2025}, Xu posed the following two questions (see \cite[Question 7.1 and Question 7.2]{Xu-2025}):


\vspace{0.2cm}

$\mathbf{Question~1.}$ For a sup-complete poset $P$, is $(P, \upsilon(P))$ strongly well-filtered?

\vspace{0.2cm}

$\mathbf{Question~2.}$
Is the Hoare power space $P_H(X)$ of a $T_0$-space $X$ strongly well-filtered?

\vspace{0.2cm}

In this paper, we first introduce a new class of $T_0$-spaces --- strong R-spaces, which are stronger than both R-spaces and strongly well-filtered spaces. Some relations among $T_2$-spaces, $T_1$-spaces, strong R-spaces, R-spaces, strongly well-filtered spaces, well-filtered spaces and sober spaces are investigated. It is proved that any sup-complete poset equipped with the upper topology is a strong R-space and the Hoare power space of a $T_0$-space is a strong R-space. Hence the upper topology on a sup-complete poset is strongly well-filtered and the Hoare power space of a $T_0$-space is strongly well-filtered, which answers the two aforementioned questions.

\section{Preliminaries}
In this section, we briefly recall some basic concepts and results about ordered structures and domain theory that will be used in the paper. For further details, we refer the reader to \cite{Abramsky-Jung-1994, GHKLMS-2003, Goubault-2013}.

For a poset $P$ and $A \subseteq P$, define $\ua A=\{x\in P: a\leq x \mbox{ for some }a\in A \}$ and $\da A=\{x \in P: x \leq a \mbox{ for some }a\in A\}$. For $x\in X$, let $\ua x=\ua\{x\}$ and $\da x=\da\{x\}$. A subset $A$ is called a \emph{lower set} (resp., an \emph{upper set}) if $A=\da A$ (resp., $A=\ua A$). Define $A^{\ua}=\{u\in P: A\subseteq \da u\}$ (the set of all upper bounds of $A$ in $P$) and $A^{\da}=\{v\in P: A\subseteq \ua v\}$ (the set of all lower bounds of $A$ in $P$). The set of all nonempty upper subsets of $P$ is denoted by $\mathbf{up}(P)$. If the set of upper bounds of $A$ has a unique smallest element, we call this element the \emph{least upper bound} of $A$ or the \emph{supremum} of $A$ and write it as $\vee A$ or sup $A$. Dually, the greatest lower bound is written as $\wedge A$ or inf $A$. We use $P^{op}$ to denote the dual poset of $P$. The set of all natural numbers is denoted by $\mathbb N$. Let $\mathbb N^{+}=\mathbb N\setminus\{0\}$. When $\mathbb N$ is regarded as a poset (in fact, a chain), the order on $\mathbb N$ is the usual order of natural numbers.  For a set $X$, let $|X|$ be the cardinality of $X$ and let $\omega=|\mathbb{N}|$. The set of all subsets of $X$ is denoted by $2^X$.

A nonempty subset $D$ of a poset $P$ is called \emph{directed} if every finite subset of $D$ has an upper bound in $D$. Dually, a nonempty subset $F$ of $P$ is said to be \emph{filtered} if every finite subset of $F$ has a lower bound in $P$. The set of all directed sets of $P$ is denoted by $\mathcal{D}(P)$. The poset $P$ is called a \emph{directed complete poset}, or \emph{dcpo} for short, provided that the least upper bound $\vee D$ of $D$ exists in $P$ for any $D\in\mathcal D(P)$. The poset $P$ is called a \emph{sup semilattice} if for any two elements $a$, $b\in P$, the supremum $\mathrm{sup}\{a, b\} = a\vee b$ exists in $P$. The poset $P$ is said to be \emph{sup-complete} if every nonempty subset $A$ has a supremum $\vee A$ in $P$. Clearly, a poset $P$ is sup-complete iff $P$ is both a sup semilattice and a dcpo.

\begin{lemma}\label{lem-inf-sup-lowerbound}
Let $P$ be a sup-complete poset and $A$ be a nonempty subset of $P$. If $A^{\downarrow}$ is nonempty, then the infimum of $A$ exists and $\wedge A=\vee A^{\downarrow}$.
\begin{proof}
Since $A^{\downarrow}\neq\emptyset$ and $P$ is sup-complete, $\vee A^{\downarrow}$ exists in $P$. For any $a\in A$, we have $A^{\downarrow}=\bigcap_{c\in A}\downarrow c\subseteq\downarrow a$, and hence $\vee A^{\downarrow}\leq a$. Thus $\vee A^{\downarrow}$ is a lower bound of $A$. Suppose $y$ is another lower bound of $A$. Then $y\in A^{\downarrow}$, whence $y\leq \vee A^{\downarrow}$. Therefore, $\vee A^{\downarrow}$ is the greatest lower bound of $A$, namely, $\wedge A=\vee A^{\downarrow}$.
\end{proof}
\end{lemma}

The notions in the following definition are well-known (cf. \cite{GHKLMS-2003}).

\begin{definition}\label{def-continuous-domain} Let $P$ be a dcpo and $s, t\in P$. We say $s$ is \emph{way below} $t$, written $s\ll t$, if for each $D\in \mathcal D(P)$ for which $\vee D$ exists in $P$, $t\leq \vee D$ implies $s\leq d$ for some $d\in D$. For $x\in P$, let $\Downarrow x = \{u\in P : u\ll x\}$ and $K(P)=\{k\in P : k\ll k\}$. Points in $K(P)$ are called \emph{compact} elements of $P$.
\begin{enumerate}[\rm (1)]
\item $P$ is called a \emph{continuous domain} if for each $x\in P$, $\Downarrow x$ is directed and $x=\vee\Downarrow x$.
\item $P$ is called an \emph{algebraic domain} if for each $x\in P$, $\da x\cap K(P)$ is directed and $x=\vee (\da x\cap K(P))$.
\end{enumerate}
\end{definition}

As in \cite{GHKLMS-2003}, the \emph{upper topology} on a poset $P$, generated by the complements of the principal ideals of $P$, is denoted by $\upsilon(P)$. A subset $U$ of a poset $P$ is \emph{Scott open} if (i) $U=\mathord{\uparrow}U$, and (ii) for any directed subset $D$ for which $\vee D$ exists, $\vee D\in U$ implies $D\cap U\neq\emptyset$. All Scott open subsets of $P$ form a topology, and we call this topology the \emph{Scott topology} on $P$ and denote it by $\sigma(P)$. The space $\Sigma P=(P, \sigma(P))$ is called the \emph{Scott space} of $P$. The upper sets of $P$ form the (\emph{upper}) \emph{Alexandroff topology} $\alpha (P)$.

The following is a key feature of the Scott topology (see \cite[Proposition II-2.1]{GHKLMS-2003}).

\begin{lemma}\label{Scott-cont-charac} Let $P, Q$ be posets and $f : P \longrightarrow Q$. Then the following two conditions are equivalent:
\begin{enumerate}[\rm (1)]
	\item $f$ is Scott continuous, that is, $f : \Sigma~\!\! P \longrightarrow \Sigma~\!\! Q$ is continuous.
	\item For any $D\in \mathcal D(P)$ for which $\vee D$ exists, $f(\vee D)=\vee f(D)$.
\end{enumerate}
\end{lemma}

For a $T_0$-space $X$, we use $\leq_X$ to represent the \emph{specialization order} of $X$, that is, $x\leq_X y$ if{}f $x\in \overline{\{y\}}$. In the following, when a $T_0$-space $X$ is considered as a poset, the order always refers to the specialization order if no other explanation is given. Let $\mathcal O(X)$ (resp., $\Gamma(X)$) be the set of all open subsets (resp., closed subsets) of $X$. Let $X^{(<\omega)}$ be the set of all nonempty finite subsets of $X$, $\mathbf{Fin}~\!\!X=\{\ua F : F\in X^{(<\omega)}\}$ and $\mathbf{Fin}^{op}~\!\!X=\{\downarrow F : F\in X^{(<\omega)}\}$. The closure of a subset $A$ in $X$ will be denoted by $\cl_X A$ (or simply by $\cl A$ if there is no ambiguity) or $\overline{A}$, and the interior of $A$ will be denoted by $\ii_X A$ or simply by $\ii A$. The space $X$ is called a \emph{$d$-space} (or \emph{monotone convergence space}) if $X$ (endowed with the specialization order) is a dcpo and $\mathcal O(X) \subseteq \sigma(X)$ (cf. \cite{GHKLMS-2003, Wyler-1981}).

A topology $\tau$ on a poset $P$ is called \emph{order}-\emph{compatible} if its specialization
order agrees with the original order on $P$, or equivalently, $\upsilon(P)\subseteq \tau\subseteq \alpha(P)$.



A nonempty subset $A$ of a $T_0$-space $X$ is said to be \emph{irreducible} if for any $F_1$, $F_2\in\Gamma(X)$, $A\subseteq F_1\cup F_2$ implies $A\subseteq F_1$ or $A\subseteq F_2$. We denote by $\ir(X)$ (resp., $\Gamma(X)$) the set of all irreducible (resp., irreducible closed) subsets of $X$. Clearly, every directed subset of $X$ (with the specialization order) is irreducible and the nonempty irreducible subsets of a poset equipped with the Alexandroff topology are exactly the directed sets of $P$ (cf. \cite[Fact 2.6]{Heckmann-Keimel-2013} or \cite[Lemma 1.2]{Hoffmann-1979}). The space $X$ is called \emph{sober}, if for any $A\in\Gamma(X)$, there is a unique point $x\in X$ such that $A=\overline{\{x\}}$. We say that $X$ is \emph{irreducible complete} if for any $A\in\ir(X)$, $\vee A$ exists in $X$. For a subset $B$ of $X$, it is easy to verify that $\vee B$ exists in $X$ iff $\vee\overline{B}$ exists in $X$, and $\vee B=\vee\overline{B}$ if they exist in $X$. So $X$ is irreducible complete iff $\vee A$ exists in $X$ for all $A\in\Gamma(X)$.

\begin{proposition} (\cite[Proposition 2.9]{Xu-Shen-Xi-Zhao-2020-2}) \label{prop-irr-complete-sober}
For any poset $P$, the space $(P,\upsilon(P))$ is sober iff it is irreducible complete.
\end{proposition}

\begin{corollary} \label{cor-sup-complete-sober}
For a sup-complete poset (especially, a complete lattice) $P$, $(P,\upsilon(P))$ is sober.
\end{corollary}

For a topological space $X$, $\mathcal{G}\subseteq 2^{X}$ and $A\subseteq X$, let $\Diamond_{\mathcal{G}}A=\{G\in \mathcal{G}: G\cap A\neq\emptyset\}$ and $\Box_{\mathcal{G}}A=\{G\in \mathcal{G}: G\subseteq A\}$. The sets $\Diamond_{\mathcal{G}}A$ and $\Box_{\mathcal{G}}A$ will be simply written as $\Diamond A$ and $\Box A$ respectively if there is no confusion. The \emph{lower Vietoris topology} on $\mathcal{G}$ is the topology that has $\{\Diamond U: U\in \mathcal{O}(X)\}$ as a subbase, and the resulting space is denoted by $P_{H}(\mathcal{G})$. If $\mathcal{G}\subseteq \ir (X)$, then $\{\Diamond_{\mathcal{G}}U: U\in \mathcal{O}(X)\}$ is a topology on~$\mathcal{G}$. The space $P_H(\Gamma(X)\setminus\{\emptyset\})$ is called the \emph{Hoare power space} or \emph{lower space} of $X$ and is denoted by $P_H(X)$ for short (cf. \cite{Schalk-1993}). The canonical mapping $\eta_X : X\rightarrow P_H(X)$, $x\mapsto \overline{\{x\}}$, is a topological embedding.

Clearly, $P_H(X)=(\Gamma(X) \setminus\{\emptyset\}, \upsilon(\Gamma(X)\setminus\{\emptyset\}))$. So by Corollary \ref{cor-sup-complete-sober} we get the following.

\begin{corollary} \label{cor-Hoare-power-space-sober}
For a $T_0$-space $X$, its Hoare power space $P_H((P,\upsilon(P)))$ is sober.
\end{corollary}

Rudin's Lemma (see \cite{Rudin-1981}) is a useful tool in non-Hausdorff topology and plays a crucial role in domain theory (see \cite{Abramsky-Jung-1994, GHKLMS-2003, Gierz-Lawson-Stralka-1983, Goubault-2013}). In this paper we will use Jung's version of Rudin's Lemma below (see \cite[Theorem 4.11]{Jung-1989} or \cite[Corollary 3.5]{Heckmann-Keimel-2013}).

\begin{lemma} (Rudin's Lemma) \label{lem-Rudin}
Let $P$ be a poset and $\mathcal{F}\subseteq P^{(<\omega)}$ such that $\{\ua F : F\in\mathcal F\}$ is a filtered family. Then there exists directed set $D\subseteq \mathop{\bigcup}\limits_{F \in \mathcal{F}} F$ such that $D\cap F\neq\emptyset$ for all $F\in \mathcal{F}$.
\end{lemma}

\begin{lemma}\label{lem-filtred-family} Suppose that $\{\ua F_d : d\in D\}$ is a filtered family of nonempty upper sets of a poset $P$,  and $A, U\subseteq P$. Then we have the following two conclusions:
\begin{enumerate}[\rm (1)]
\item If $F_d\cap \da A\neq\emptyset$ for any $d\in D$, then $\{\ua (F_d\cap \da A) : d\in D\}$ is filtered.
\item If $F_d\not\subseteq \ua U$ for any $d\in D$, then $\{\ua (F_d\setminus \ua U) : d\in D\}$ is filtered.
\end{enumerate}
\end{lemma}
\begin{proof} (1): First, we show that $\ua (F\cap \da C)=\ua (\ua F\cap \da C)$ for any $F, C\subseteq P$. Obviously, $\ua (F\cap \da C)\subseteq\ua (\ua F\cap \da C)$. Conversely, if $t\in \ua (\ua F\cap \da C)$, then there is $s\in \ua F\cap \da C$ such that $s\leq t$, whence there exists $f\in F$ with $f\leq s$. So $f\in F\cap \da C$ and $f\leq t$. Thus $t\in \ua (F\cap \da C)$. Therefore, $\ua (\ua F\cap \da C)\subseteq \ua (F\cap \da C)$, and hence $\ua (F\cap \da C)=\ua (\ua F\cap \da C)$. As $\{\ua F_d : d\in D\}$ is filtered, $\{\ua (F_d\cap \da A) : d\in D\}=\{\ua (\ua F_d\cap \da A) : d\in D\}$ is also filtered.

(2): Let $B=X\setminus \ua U$. Then $B=\da B$ and $F_d\cap \da B=F_d\setminus B=F_d\cap (X\setminus \ua U)=F_d\setminus \ua U\neq\emptyset$ for any $d\in D$. Hence by (1) we get that $\{\ua (F_d\setminus \ua U) : d\in D\}=\{\ua (F_d\cap B) : d\in D\}$ is filtered.
\end{proof}


\section{Strong $d$-spaces and strongly well-filtered spaces}

A subset $A$ of a $T_0$-space $X$ is called \emph{saturated} if $A$ equals the intersection of all open sets containing it (equivalently, $A$ is an upper set in the specialization order). We use $\mathsf{K}(X)$ to denote the set of all nonempty compact saturated subsets of $X$ and endow it with the Smyth order $\sqsubseteq:~K_1\sqsubseteq K_2$ iff $K_{2}\subseteq K_{1}$. The space $X$ is called \emph{well-filtered} if for any filtered family $\mathcal{K}\subseteq \mathord{\mathsf{K}}(X)$ and any open set $U$, $\bigcap\mathcal{K}{\subseteq} U$ implies $K{\subseteq} U$ for some $K{\in}\mathcal{K}$. The space $X$ is said to be \emph{coherent} if the intersection of any two compact saturated sets is compact.

The following implications are well-known (see, e.g., \cite[Theorem II-1.21]{GHKLMS-2003} or \cite[Theorem 6.6 and Corollary 7.12]{Xu-Shen-Xi-Zhao-2020-2}):

\begin{center}
sober spaces $\Rightarrow$ well-filtered spaces $\Rightarrow$ $d$-spaces.
\end{center}

A poset $P$ is said to be \emph{Noetherian} if it satisfies the \emph{ascending chain condition} ($\mathrm{ACC}$ for short): every ascending chain has a greatest member.

\begin{proposition}\label{Alexandroff-topology-sober} (\cite[Proposition 11]{Lawson-Xu-2024-2}).
	For any poset $P$, the following conditions are equivalent:
	\begin{enumerate}[\rm (1)]
		\item $P$ is Noetherian.
        \item $P$ is a dcpo such that every element of $P$ is compact (i.e., $x\ll x$ for all $x\in P$).
		\item $P$ is a dcpo and $\alpha(P)=\sigma(P)$.
       \item $(P,\alpha(P))$ is sober.
		\item $(P,\alpha(P))$ is well-filtered.
		\item $(P,\alpha(P))$ is a $d$-space.
	\end{enumerate}
\end{proposition}

As a strengthened version of well-filtered spaces (resp., $d$-spaces), the following two notions were introduced in \cite{Xu-2025, Xu-Zhao-2020}.

\begin{definition}(\cite[Definition 4.1]{Xu-2025}) \label{def-strong-WF} A $T_0$-space $X$ is called \emph{strongly well}-\emph{filtered} if for any filtered family $\{K_d : d\in D\}\subseteq \mathsf{K}(X)$, $K\in \mathsf{K}(X)$ and $U\in \mathcal O(X)$, $\bigcap_{d\in D}K_d\cap K\subseteq U$ implies $K_d\cap K\subseteq U$ for some $d\in D$.
\end{definition}

\begin{definition}\label{strong $d$-space} (\cite[Definition 3.18]{Xu-Zhao-2020}) A $T_0$-space $X$ is called a \emph{strong} $d$-\emph{space} if for any $D\in \mathcal D(X)$, $x\in X$ and $U\in \mathcal O(X)$, $\bigcap_{d\in D}\ua d\cap \ua x\subseteq U$ implies $\ua d\cap \ua x\subseteq U$ for some $d\in D$.
\end{definition}

\begin{proposition} (\cite[Proposition 4.3]{Xu-2025}) \label{prop-WF-coherent-strongly-WF} \emph{(1)} Every strongly well-filtered space is well-filtered.

\noindent \emph{(2)} Every strongly well-filtered space is a strong $d$-space.

\noindent \emph{(3)} Every coherent well-filtered space is strongly well-filtered.

\noindent \emph{(4)} Every $T_2$-space is coherent and sober and hence strongly well-filtered.
\end{proposition}

\begin{theorem} (\cite[Theorem 4.8]{Xu-2025}) \label{theor-d-space-intersection-strongly-WF}
Let $X$ be a $d$-space such that $\downarrow(A\cap K)$ is closed for any nonempty closed set $A$ and $K\in\mathsf K(X)$. Then $X$ is strongly well-filtered.
\end{theorem}

\begin{proposition} (\cite[Proposition 4.3]{Xu-2025}) \label{prop-complete-lattice-strongly-WF}
For a complete lattice $L$, $(L, \upsilon(L))$ is strongly well-filtered.
\end{proposition}

\section{Main results}

\begin{definition}  (\cite[Definition 30 and Remark 38]{Lawson-Xu-2024-2}) \label{def-R-space} A $T_0$-space $X$ is called an \emph{R}-\emph{space} if for any nonempty family $\{\mathord{\uparrow}F_{i}:i\in I\}\subseteq \mathbf{Fin}~\!\! X$ and any $U\in \mathcal O(X)$, $\bigcap_{i\in I}\mathord{\uparrow}F_{i}\subseteq U$ implies  $\bigcap_{i\in I_{0}}\mathord{\uparrow}F_{i}\subseteq U$ for some $I_{0}\in I^{(<\omega)}$.
\end{definition}

The R-spaces were first introduced in \cite[Definition 10.2.11]{Xu-2016-2} (see also \cite{Wen-Xu-2018}), and they have been systematically studied in \cite{Lawson-Xu-2024-2, Xu-2016-2}.

As a strengthened version of R-spaces, we introduce the following notion.

\begin{definition}\label{def-strong-R-space} A $T_0$-space $X$ is called a \emph{strong} \emph{R}-\emph{space} if for any nonempty family $\{K_{i} : i\in I\}\subseteq \mathsf{K}(X)$ and any $U\in \mathcal O(X)$, $\bigcap_{i\in I}K_i\subseteq U$ implies  $\bigcap_{i\in I_{0}}K_i\subseteq U$ for some $I_{0}\in I^{(<\omega)}$.
\end{definition}

Clearly, every strong R-space is an R-space and the Sierpi\'{n}ski space $\Sigma 2$ is a strong R-space but not $T_1$.

\begin{proposition}\label{prop-T2-is-strong-R-space}  \noindent \emph{(1)} Every strong R-space is strongly well-filtered.

\noindent \emph{(2)} Every coherent well-filtered space is a strong R-space.

\noindent \emph{(3)} Every $T_2$-space is a coherent sober space and hence a strong R-space.
\end{proposition}
\begin{proof} (1): Suppose that $\{K_d : d\in D\}\subseteq \mathsf{K}(X)$ is a filtered family, $K\in \mathsf{K}(X)$ and $U\in \mathcal O(X)$ satisfying $\bigcap_{d\in D}K_d\cap K\subseteq U$. Select $i\not\in D$ and let $K_i=K$. Setting $I=D\cup\{i\}$, we conclude there exists a finite subset $I_0$ of $I$ such that $\bigcap\limits_{j\in I_0}K_j\subseteq U$. As $\{K_d : d\in D\}$ is filtered, there is $d_0\in D$ such that $K_{d_0}\subseteq \bigcap\limits_{d\in I_0\setminus \{i\}}K_d$. Then $K_{d_0}\cap K\subseteq \bigcap\limits_{j\in I_0}K_j\subseteq U$. Hence $X$ is strong well-filtered.

(2): Assume that $X$ is a coherent well-filtered space. Let $\{K_i : i\in I\}\subseteq \mathsf{K}(X)$ be a nonempty family and $U\in \mathcal O(X)$ satisfying $\bigcap_{i\in I} K_i\subseteq U$. If there is $I_0\in I^{(<\omega)}$ with $\bigcap\limits_{i\in I_0}K_i=\emptyset$, then $\bigcap\limits_{i\in I_0}K_i\subseteq U$. Now we assume that $\bigcap\limits_{i\in J}K_i\neq \emptyset$ for all $J\in I^{(<\omega)}$. Then by the coherence of $X$, $\{\bigcap\limits_{i\in J}K_i : J\in I^{(<\omega)}\}\subseteq \mathsf{K}(X)$ is a filtered family, and $\bigcap\limits_{J\in I^{(<\omega)}}\bigcap\limits_{i\in J}K_i=\bigcap_{i\in I} K_i\subseteq U$. AS $X$ is well-filtered, there is $I_0\in I^{(<\omega)}$ such that $\bigcap\limits_{i\in I_0}K_i\subseteq U$. Thus $X$ is strong R-space.

(3): Let $X$ be a $T_2$-space. Then it is easy to verify that $X$ is sober (see, e.g., \cite[Proposition 8.2.12]{Goubault-2013}) and hence well-filtered. As $X$ is $T_2$, $\mathsf{K}(X)\subseteq \Gamma(X)$ (see \cite[Theorem 3.1.8]{Engelking-1989}). For any $K_1, K_2\in   \mathsf{K}(X)$, $K_1\cap K_2$, as the intersection of a compact subset $K_1$ and a closed subset $K_2$, is a compact set of $X$. So $X$ is coherent. Therefore, by (2) $X$ is a strong R-space.
\end{proof}

The following example shows that a coherent well-filtered space may not be sober, and hence a strong R-space may not be sober.

\begin{example}\label{exam-X-coc-strongly-WF-not-sober}
	Let $X$ be a uncountably infinite set and $X_{coc}$ the space equipped with \emph{the co-countable topology} (the empty set and the complements of countable subsets of $X$ are open). Then
\begin{enumerate}[\rm (a)]
    \item $\mathcal C(X_{coc})=\{\emptyset, X\}\cup X^{(\leqslant\omega)}$ and $X_{coc}$ is $T_1$.

\item $\ir_c(X_{coc})=\{X\}\cup \{\{x\} : x\in X\}$ and hence $X_{coc}$ is non-sober.

        \item $\mk (X_{coc})=X^{(<\omega)}\setminus \{\emptyset\}$.

        Clearly, every finite subset is compact. Conversely, if $C\subseteq X$ is infinite, then $C$ has an infinite countable subset $\{c_n : n\in\mn\}$. Let $C_0=\{c_n : n\in\mn\}$ and $U_m=(X\setminus C_0)\cup \{c_m\}$ for each $m\in \mn$. Then $\{U_n : n\in\mn\}$ is an open cover of $C$, but has no finite subcover. Hence $C$ is not compact. Thus $\mk (X_{coc})=X^{(<\omega)}\setminus \{\emptyset\}$.

    \item $X_{coc}$ is well-filtered and coherent.

       By (c), $X_{coc}$ is coherent. Suppose that $\{F_d : d\in D\}\subseteq \mk (X_{coc})$ is a filtered family and $U\in \mathcal O(X_{coc})$ with $\bigcap_{d\in D}F_d\subseteq U$. As $\{F_d : d\in D\}$ is filtered and all $F_d$ are finite, $\{F_d : d\in D\}$ has a least element $F_{d_0}$, and hence $F_{d_0}=\bigcap_{d\in D}F_d\subseteq U$. Thus $X_{coc}$ is well-filtered.

      \item $X_{coc}$ is a strong R-space and hence a strong well-filtered space by (d) and Proposition \ref{prop-T2-is-strong-R-space}.
\end{enumerate}
\end{example}

Figure 1 shows certain relations of some spaces lying between $d$-spaces and $T_2$-spaces (cf. \cite[Figure 1]{Lawson-Xu-2024-2}).

\begin{figure}[ht]
	\centering
	\includegraphics[height=3.0cm,width=10.0cm]{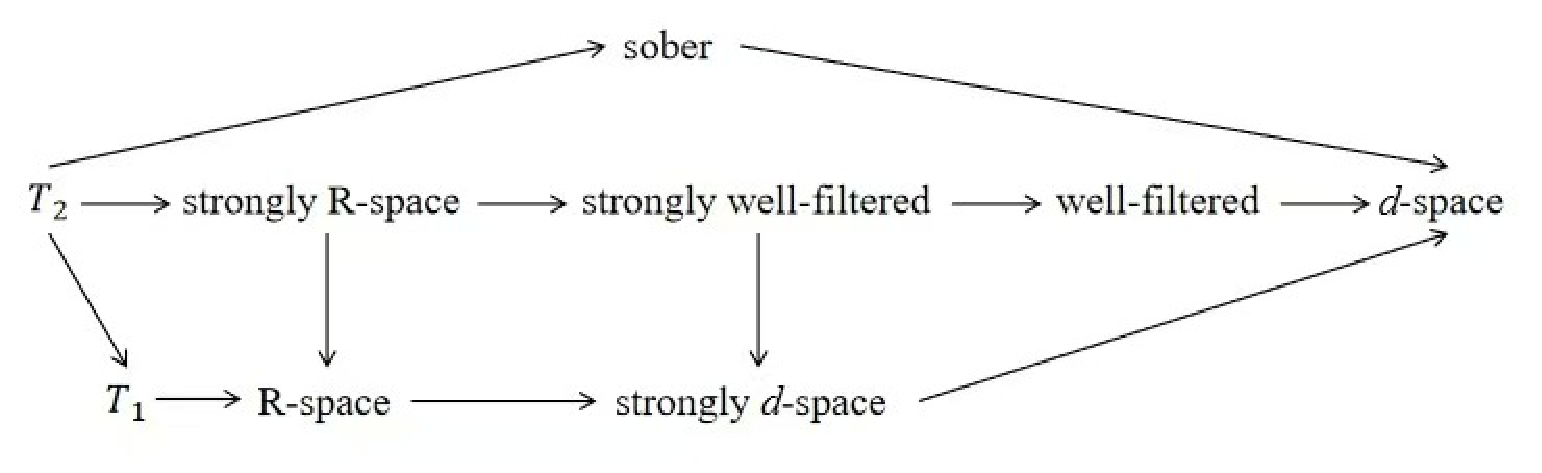}
	\caption{Relations of some spaces lying between $d$-spaces and $T_2$-spaces}
\end{figure}

\begin{theorem}\label{theor-compact-bigcap}
 For be a sup-complete poset $P$, $(P, \upsilon(P))$ is a strong R-space and hence a strongly well-filtered space.
\end{theorem}
\begin{proof}
Suppose that $\{K_i:~i\in I\}\subseteq\mathsf{K}((P,~\upsilon(P)))\setminus\{\emptyset\}$ and $U\in\upsilon(P)$ satisfying $\bigcap\limits_{i\in I}K_i\subseteq U$. We will show that there exists $I_0\in I^{(<\omega)}$ such that $\bigcap\limits_{i\in I_0}K_i\subseteq U$.

{\bf Case 1:} $U=P$.

Then for any $i\in I$, we have $K_i\subseteq U$.

{\bf Case 2:} $U\neq P$.

As $P$ is sup-complete, it has a greatest element $\top=\vee P$. For any $i\in I$, as $\emptyset\neq K_i=\uparrow K_i$, we have $\top\in K_i$. Then $\top\in \bigcap\limits_{i\in I}K_i\subseteq U$, whence $U\neq\emptyset$. Therefore, there exists a nonempty family $\{F_j:~j\in J\}\subseteq P^{(<\omega)}$ such that $U=\bigcup\limits_{j\in J}(P\setminus\downarrow F_j)=P\setminus\bigcap\limits_{j\in J}\downarrow F_j$. As $P\neq U$, we have $\bigcap\limits_{j\in J}\downarrow F_j\neq\emptyset$. Hence for any $S\in P^{(<\omega)}$, $\emptyset \neq \bigcap\limits_{j\in S}\downarrow F_j=\bigcup\limits_{\varphi\in \prod\limits_{j\in S}F_j}\bigcap\limits_{j\in S}\downarrow \varphi (j)$. Let $\Phi_S=\{\varphi\in \prod\limits_{j\in S}F_j : \bigcap\limits_{j\in S}\downarrow \varphi(j)\neq \emptyset\}$. Then $\Phi_S\neq \emptyset$, and for any $\varphi \in \Phi_S$, by Lemma \ref{lem-inf-sup-lowerbound} we get that $\bigwedge\limits_{j\in S}\varphi(j)$ exists in $P$, whence $\bigcap\limits_{j\in S}\downarrow \varphi(j)=\downarrow \bigwedge\limits_{j\in S}\varphi(j)=\downarrow \wedge \varphi(S)$. For $S\in P^{(<\omega)}$, let $F_S=\{\wedge\varphi(S) : \varphi\in \Phi_S\}$. Then $\emptyset\neq F_S\in P^{(<\omega)}$ and $\bigcap_{j\in S}\downarrow F_j=\downarrow F_S$. So $\{\downarrow F_S : \varphi\in J^{(<\omega)}\}\subseteq \mathbf{Fin}^{op}~\!P$ is a filtered family, and $\bigcap\limits_{S\in J^{(<\omega)}}\downarrow F_S=\bigcap\limits_{j\in J} \downarrow F_j= P\setminus U \subseteq P\setminus \bigcap\limits_{i\in I}K_i=\bigcup\limits_{i\in I} (P\setminus K_i)$. Let $W=\bigcup\limits_{i\in I} (P\setminus K_i)$. Then $W$ is a nonempty lower subset of $P$.

Now we claim that there is $S_0\in J^{(<\omega)}$ such that $\downarrow F_{S_0}\subseteq W$. Assume, on the contrary, that $\downarrow F_S\not\subseteq W$ for any $S\in J^{(<\omega)}$. Then by Lemma \ref{lem-filtred-family} (applying it to the dual poset $P^{op}$), $\{\downarrow (F_S\setminus W) : S\in J^{(<\omega)}\}\subseteq \mathbf{Fin}^{op}~\!P$ is filtered. By Lemma \ref{lem-Rudin} (also applying it to the dual poset $P^{op}$), there is a filtered subset $E\subseteq \bigcup\limits_{S\in J^{(<\omega)}}(F_S\setminus W)$ such that $E\cap (F_S\setminus W)\neq\emptyset$ for all $S\in J^{(<\omega)}$. If $\bigcap\limits_{e\in E}\downarrow e=\emptyset$, then $K_i\subseteq P\setminus \bigcap\limits_{e\in E}\downarrow e=\bigcup\limits_{e\in E}(P\setminus \downarrow e)$ for any $i\in I$. Hence $\{P\setminus \downarrow e : e\in E\}\subseteq \upsilon(P)$ is a directed open cover of $K_i$. Therefore, there is $e\in E$ such that $K_i\subseteq P\setminus \downarrow e$, and hence $e\in P\setminus K_i\subseteq W$, which is a contradiction with $E\subseteq \bigcup\limits_{S\in J^{(<\omega)}}(F_S\setminus W)\subseteq P\setminus W$. Thus $\bigcap\limits_{e\in E}\downarrow e\neq\emptyset$. By Lemma \ref{lem-inf-sup-lowerbound}, $\wedge E$ exists in $P$. So $\downarrow \wedge E=\bigcap\limits_{e\in E}\downarrow e\subseteq \bigcap\limits_{S\in J^{(<\omega)}}\downarrow F_S\subseteq W=\bigcup\limits_{i\in I}(P\setminus K_i)$. Then there is $i_0\in I$ such that $\downarrow \wedge E\subseteq P\setminus K_{i_0}$, and hence  $K_{i_0}\subseteq \bigcup\limits_{e\in E}$. As $K_{i_0}\in \mathsf{K}((P, \upsilon(P)))$ and $E$ is filtered, there is $e\in E$ such that $K_{i_0}\subseteq P\setminus \downarrow e$, whence $e\in P\setminus K_{i_0}\subseteq W$, which is in contradiction with $E\subseteq P\setminus W$. Thus there is $S_0\in J^{(<\omega)}$ such that $\bigcap\limits_{j\in S_0}\downarrow F_j=\downarrow F_{S_0}\subseteq W=\bigcup\limits_{i\in I}(P\setminus K_i)$. As $F_{S_0}$ is finite, there is $I_0\in I^{(<\omega)}$ such that $\downarrow F_{S_0}\subseteq \bigcup\limits_{i\in I_0}(P\setminus K_i)=P\setminus\bigcap\limits_{i\in I_0} K_i$, and hence $\bigcap\limits_{i\in I_0} K_i\subseteq P\setminus \downarrow F_{S_0}=P\setminus \bigcap\limits_{j\in S_0}\downarrow F_j=\bigcup\limits_{j\in S_0}(P\setminus \downarrow F_j)\subseteq U$. This completes the proof that $(P, \upsilon(P))$ is a strong R-space.
\end{proof}

For a $T_0$-space $X$, $\Gamma(X)\setminus \{\emptyset\}$ is clearly a sup-complete poset and $P_H(X)=(\Gamma(X) \setminus\{\emptyset\}, \upsilon(\Gamma(X)\setminus\{\emptyset\}))$. So by Theorem \ref{theor-compact-bigcap} we get the following.

\begin{theorem}\label{cor-Hower-swf}
For a $T_0$-space $X$, its Hoare power space $P_H(X)$ is a strong R-space and hence a strongly well-filtered space.
\end{theorem}

\begin{proposition}\label{prop-sup-complete-coherent}
For a sup-complete poset $P$, $(P, \upsilon(P))$ is coherent.
\end{proposition}
\begin{proof}
Suppose $K_1, K_2\in\mathsf{K}((P, \upsilon(P)))$. We will show that $K_1\cap K_2\in\mathsf{K}((P, \upsilon(P)))$. As $P$ is sup-complete, it has a greatest element $\top=\vee P$. Clearly, $\top\in K_1\cap K_2$. Hence $K_1\cap K_2$ is a nonempty saturated subset of $P$.

$\mathbf{Case~1:}$ $K_1=P$ or $K_2=P$.

Then $K_1\cap K_2=K_2\in\mathsf{K}((P, \upsilon(P)))$ or $K_1\cap K_2=K_1\in\mathsf{K}((P, \upsilon(P)))$.

$\mathbf{Case~2:}$ $K_1\neq P\neq K_2$.

We use Alexander Subbase Theorem (cf. \cite[Problem 3.12.2]{Engelking-1989}) to prove that $K_1\cap K_2$ is compact in $(P, \upsilon(P))$. Let $\{P\setminus\downarrow x_i: i\in I\}$ be an open cover of $K_1\cap K_2$ in $(P, \upsilon(P))$, that is, $K_1\cap K_2\subseteq\bigcup\limits_{i\in I}(P\setminus\downarrow x_i)=$ $P\setminus\bigcap\limits_{i\in I}\downarrow x_i$. Then we need to show that there exists $I_0\in I^{(<\omega)}$ such that $K_1\cap K_2\subseteq\bigcup\limits_{i\in I_0}(P\setminus\downarrow x_i)$.

(C1) $\bigcap_{i\in I}\downarrow x_i=\emptyset$.

 Then $K_1\subseteq P=P\setminus\bigcap\limits_{i\in I}\downarrow x_i=\bigcup\limits_{i\in I}(P\setminus \downarrow x_i)$. By $K_1\in\mathsf{K} ((P, \upsilon(P)))$, there is $I_0\in I^{(<\omega)}$ such that $K_1\subseteq\bigcup\limits_{i\in I_0}(P\setminus\downarrow x_i)$, whence $K_1\cap K_2\subseteq\bigcup\limits_{i\in I_0}P\setminus\downarrow x_i$.

(C2) $\bigcap_{i\in I}\downarrow x_i\neq\emptyset$.

 Then by Lemma \ref{lem-inf-sup-lowerbound}, $\bigwedge\limits_{i\in I}x_i$ exists in $P$, whence $\downarrow\wedge_{i\in I} x_i=\bigcap_{i\in I}\downarrow x_i$. So $K_1\cap K_2\subseteq P\setminus\bigcap\limits_{i\in I}\downarrow x_i=P\setminus\downarrow\bigwedge\limits_{i\in I} x_i$ or, equivalently, $\downarrow\bigwedge\limits_{i\in I} x_i\subseteq P\setminus(K_1\cap K_2)=(P\setminus K_1)\cup (P\setminus K_2)$. Without loss of generality, we can assume $\downarrow\bigwedge\limits_{i\in I} x_i\subseteq P\setminus K_1$. Then $K_1\subseteq P\setminus \subseteq\bigcup\limits_{i\in I}(P\setminus\downarrow x_i)$. By the compactness of $K_1$, there exists an $I_0\in I^{(<\omega)}$ such that $K_1\subseteq\bigcup\limits_{i\in I_0}(P\setminus\downarrow x_i)$, whence $K_1\cap K_2\subseteq\bigcup\limits_{i\in I_0}(P\setminus\downarrow x_i)$. By Alexander Subbase Theorem, we have $K_1\cap K_2\in\mathsf{K}((P, \upsilon(P)))$. Thus $(P, \upsilon(P))$ is coherent.
\end{proof}

\begin{remark}\label{rem-also-get-main-results} By Corollary \ref{cor-sup-complete-sober}, Proposition \ref{prop-T2-is-strong-R-space} and Proposition \ref{prop-sup-complete-coherent}, we also obtain Theorem \ref{theor-compact-bigcap} and Theorem \ref{cor-Hower-swf}.
\end{remark}

Finally, we give three examples to illustrate that when $P$ is a dcpo but not sup-complete, $(P, \upsilon(P))$ may not be well-filtered (and hence not strongly well-filtered) or may not be coherent.

\begin{example}\label{exam-X-cof-strong-d-space-not-WF}
	Let $X$ be a countably infinite set and $X_{cof}$ the space equipped with the \emph{co-finite topology} (the empty set and the complements of finite subsets of $X$ are open). Then
\begin{enumerate}[\rm (a)]
\item $X$ (with the discrete order) is a dcpo (in fact, it is a countable Noetherian dcpo).
\item $(X, \upsilon(X))=X_{cof}$.
    \item $\mathcal C(X_{cof})=\{\emptyset, X\}\cup X^{(<\omega)}$, $X_{cof}$ is $T_1$ and hence a strong $d$-space.
    \item $\mk (X_{cof})=2^X\setminus \{\emptyset\}$.
    \item  $\mathcal O(X_{cof})$ is countable, whence $X_{cof}$ is second-countable.
    \item $\mk (X_{cof})=2^X\setminus \{\emptyset\}$. So $X_{cof}$ is locally compact.
\item $X_{cof}$ is not well-filtered.

Let $\mathcal K=\{X\setminus F: F\in X^{(<\omega)}\}$. Then $\mathcal K$ is a filtered family of saturated compact subsets of $X_{cof}$ and $\bigcap \mathcal K=\emptyset$, but $X\setminus F\neq\emptyset$ for every $ F\in X^{(<\omega)}$. Thus $X_{cof}$ is not well-filtered.
\end{enumerate}
\end{example}

\begin{example}\label{exam-Johnstone-not-supsemlattice-not-swf} (Johnstone's dcpo)
Let $\mathbb{J}=\mathbb{N}^{+}\times (\mathbb{N}^{+}\cup \{\infty\})$ with ordering defined by $(j, k)\leq (n, m)$ if{}f $j = n$ and $k \leq m$, or $m =\infty$ and $k\leq n$ (see Figure 2). $\mathbb{J}$ is a well-known dcpo constructed by Johnstone in \cite{Johnstone-1981}, which is the first given dcpo whose Scott space is non-sober. Clearly, $\mathbb{J}_{max}=\{(n, \infty) : n\in\mathbb{N}^{+} \}$ is the set of all maximal elements of $\mathbb{J}$.

\begin{figure}[ht]
	\centering
	\includegraphics[height=4.5cm,width=5.5cm]{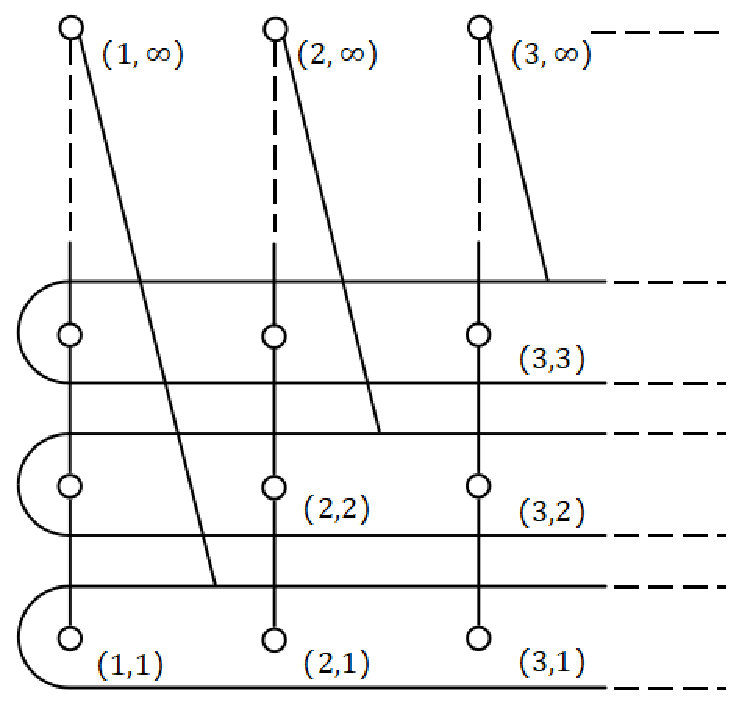}
	\caption{Johnstone's dcpo $\mathbb{J}$}
\end{figure}


The following four conclusions about $\Sigma~\!\!\mathbb{J}$ are known (see, for example, \cite[Example 3.1]{Lu-Li-2017} and \cite[Lemma 3.1]{Miao-Li-Zhao-2021}):
\begin{enumerate}[\rm (i)]
\item $\mathcal{D}(\mathbb{J})=\{D\subseteq \mathbb{J} : D \mbox{~has a largest element or there is~} n\in\mathbb{N}^{+} \mbox{~such that~} |D|=\omega \mbox{~and~}D\subseteq \{n\}\times \mathbb{N}^{+}\}$. So $\mathbb{J}$ is a dcpo.
\item $\ir_c (\Sigma~\!\!\mathbb{J})=\{\overline{\{x\}}=\da x : x\in \mathbb{J}\}\cup \{\mathbb{J}\}$.
\item $\mathsf{K}(\Sigma~\!\!\mathbb{J})=(2^{\mathbb{J}_{max}} \setminus \{\emptyset\})\bigcup \mathbf{Fin}~\!\mathbb{J}$.
\item $\Sigma~\!\!\mathbb{J}$ is not well-filtered and hence non-sober (cf. \cite[Exercise 8.3.9]{Goubault-2013}).

\end{enumerate}

Considering the space $(\mathbb{J}, \upsilon(\mathbb{J}))$, we have the following conclusions.

\begin{enumerate}[\rm (a)]
\item $(\mathbb{J}, \upsilon(\mathbb{J}))$ is not irreducible complete and hence $\mathbb J$ is not sup-complete.

By (ii) we have $\mathbb{J}\in \ir_c(\Sigma~\!\!\mathbb{J})\subseteq \ir ((\mathbb{J}, \upsilon(\mathbb{J}))$. As $\mathbb{J}$ has no greatest element, $\vee \mathbb{J}$ does not exist in $\mathbb{J}$. Thus $(\mathbb{J}, \upsilon(\mathbb{J}))$ is not irreducible complete, and hence $\mathbb J$ is not sup-complete.

\item $\mathsf{K}(\Sigma~\!\!\mathbb{J})=(2^{\mathbb{J}_{max}} \setminus \{\emptyset\})\bigcup \mathbf{Fin}~\!\mathbb{J}\subseteq \mathsf{K}((\mathbb{J}, \upsilon(\mathbb{J})))$ by (iii) and $\upsilon(\mathbb{J})\subseteq \sigma(\mathbb{J})$


\item $(\mathbb{J}, \upsilon(\mathbb{J}))$ is not well-filtered and hence not strongly well-filtered.

Let $\mathcal K=\{\mathbb{J}_{max}\setminus F : F\in (\mathbb{J}_{max})^{(<\omega)}\}$. Then by (b), $\mathcal K\subseteq \mathsf{K}((\mathbb{J}, \upsilon(\mathbb{J}))$ is a filtered family and $\bigcap\mathcal{K}=\bigcap_{F\in (\mathbb{J}_{max})^{(<\omega)}} (\mathbb{J}_{max}\setminus F)=\mathbb{J}_{max}\setminus \bigcup (\mathbb{J}_{max})^{(<\omega)})=\emptyset$, but $\mathbb{J}_{max}\setminus F\neq\emptyset$ for all $F\in (\mathbb{J}_{max})^{(<\omega)}$. Therefore, $(\mathbb{J}, \upsilon(\mathbb{J}))$ is not well-filtered, and hence it is not strong well-filtered.
\end{enumerate}
\end{example}

\begin{example}\label{exam-not-supsemlattice-not-coherent}
Let $P=\{b_n: n\in\mathbb N^{+}\}\cup\{\bot_1, \bot_2\}$ and define a partial order $\leq$ on $P$ as follows (see Figure 3):
\begin{enumerate}[\rm (1)]
\item $b_{n+1}< b_n$ for any $n\in\mathbb N^{+}$,
\item $\bot_1<b_n, \bot_2<b_n$ for all $n\in\mathbb N^{+}$, and
\item $\bot_1$ and $\bot_2$ are incomparable.
\end{enumerate}
\begin{figure}[ht]
	\centering
	\includegraphics[height=4.5cm,width=4.5cm]{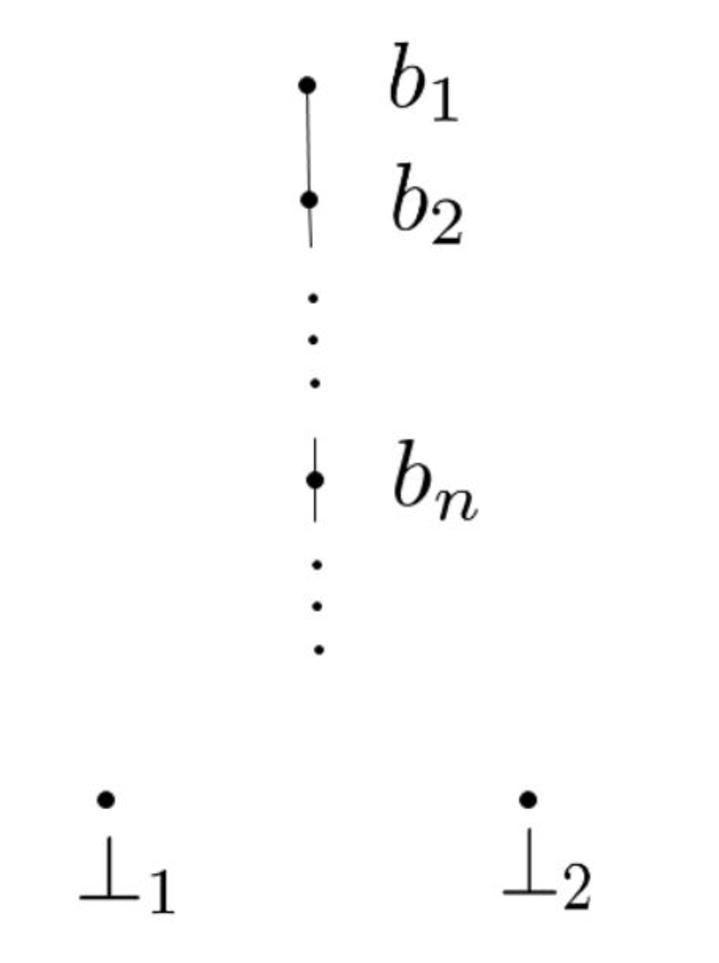}
	\caption{The poset $P$ in Example \ref{exam-not-supsemlattice-not-coherent}}
\end{figure}

\noindent Consider the upper topology space $(P, \upsilon(P))$. Then

\begin{enumerate}[\rm (a)]
\item $P$ is Noetherian, and hence $P$ is an algebraic domain and $\sigma(P)=\alpha(P)$ by Proposition \ref{Alexandroff-topology-sober}.
\item $\upsilon(P)=\alpha(P)=\{\uparrow b_n : n\in \mathbb{N}^+\}\cup\{\uparrow \top_1, \uparrow\top_2 \}\cup\{\{b_n: n\in\mathbb{N}^+\}\}\cup\{P, \emptyset\}$.

Suppose that $\emptyset\neq A=\downarrow A$. We show that $A$ is closet in $(P, \upsilon(P))$. If $A\setminus \{\bot_1, \bot_2\}\}$, then clearly $A\in \Gamma ((P, \upsilon(P))$. Now we assume that $A\cap \{b_n : n\in\mathbb{N}^+\}\neq\emptyset$. Let $m=\mathrm{min}\{n\in  n\in\mathbb{N}^+ : b_n\in A\}$. Then $A=\downarrow A=\downarrow b_m\in \Gamma ((P, \upsilon(P))$. Hence $\upsilon(P)=\alpha(P)=\mathbf{up}(P)=\{\uparrow b_n : n\in \mathbb{N}^+\}\cup\{\uparrow \top_1, \uparrow\top_2 \}\cup\{\{b_n: n\in\mathbb{N}^+\}\}\cup\{P, \emptyset\}$.

\item $\ir((P, \upsilon(P)))=\ir((P, \alpha(P)))=\mathcal D(P)=2^{P}\setminus\{\emptyset, \{\bot_1, \bot_2\}\}$ by (a) and (b).


\item $(P, \upsilon(P))$ is irreducible complete by (a) and (c).
\item $(P, \upsilon(P))$ is sober by (d) and Proposition \ref{prop-irr-complete-sober}.
\item $P$ is not a sup-semilattice ($\bot_1\vee\bot_2$ does not exist in $P$) and hence it is not a sup-complete poset.
\item $\mathsf{K}((P, \upsilon(P)))=\mathsf{K}((P, \alpha(P)))=\mathbf{Fin}~\!\!P=\mathbf{up}(P)\setminus\{\{b_n:~n\in\mathbb N^{+}\}\}=\{P, \uparrow\top_1, \uparrow\top_2\}\cup\{\uparrow a_n:~n\in\mathbb N^{+}\}$.

We only need to show that $\{b_n:~n\in\mathbb N^{+}\}\notin\mathsf{K}((P, \upsilon(P)))$. In fact, $\{b_n:~n\in\mathbb N^{+}\}\subseteq\bigcup\limits_{m\in\mathbb N^{+}}(P\setminus\downarrow b_m)=\bigcup\limits_{l\in\mathbb N^{+}, l\geq 2 }\{b_1,~b_2,~\cdots,~b_{l-1}\}$. But  $\{b_n:~n\in\mathbb N^{+}\}\subseteq P\setminus\downarrow b_m$ for no $m\in\mathbb N^{+}$. Hence $\{b_n:~n\in\mathbb N^{+}\}\notin\mathsf{K}((P, \upsilon(P)))$.

\item $(P, \upsilon(P))$ is not coherent.

 $\uparrow\bot_1, \uparrow \bot_2\in\mathsf{K}((P, \upsilon(P)))$, but $\uparrow\bot_1\cap\uparrow\bot_2=\{b_n:~n\in\mathbb N^{+}\}\notin\mathsf{K}((P, \upsilon(P)))$ by (g). So $(P, \upsilon(P))$ is not coherent.
\end{enumerate}
\end{example}



\section{Conclusion}

In this paper, we have introduced a new kind of $T_0$-spaces --- strong R-spaces,  lying between strongly well-filtered spaces and $T_2$-spaces. The strong R-spaces are closely related to R-spaces, strongly well-filtered spaces and
well-filtered spaces. By introducing and studying strong R-spaces, we have provided some finer
links between $d$-spaces and $T_2$-spaces and given answers to two questions recently posed by Xu.

The R-spaces and strongly well-filtered spaces have been proven to be two important classes of $T_0$-spaces (cf. \cite{Lawson-Xu-2024-2, Wen-Xu-2018, Xu-2016-2, Xu-2025}). Our study in this paper indicates that strong R-spaces may deserve further investigation.

\end{document}